\documentclass[11pt]{article}
\usepackage[a4paper, total={17cm, 25cm}]{geometry}

\usepackage[colorinlistoftodos, textsize=tiny]{todonotes}

\usepackage[utf8]{inputenc}

\usepackage{amsfonts}
\usepackage{amsmath}
\usepackage{amsthm}
\usepackage{amssymb}
\usepackage{mathtools}
\usepackage{xcolor}

\usepackage[normalem]{ulem} 

\usepackage{hyperref}

\usepackage{biblatex}
\newtheorem{theorem}{\bf Theorem}[section]
\newtheorem{lemma}[theorem]{\bf Lemma}
\newtheorem{prop}[theorem]{\bf Proposition}

\theoremstyle{definition}
\newtheorem{nota}[theorem]{\bf Notation}

\newtheorem{remark}[theorem]{\bf Remark}
\newtheorem{defi}[theorem]{\bf Definition}
\newtheorem{cond}[theorem]{\bf Condition}

\newcommand{\bbE}{\mathbb E}
\newcommand{\bbR}{\mathbb R}
\newcommand{\bbP}{\mathbb P}
\newcommand{\eps}{\varepsilon}
\newcommand{\1}{\mathbf 1}

\renewcommand{\vec}[1]{\boldsymbol{#1}}

\DeclarePairedDelimiter\norm{\lVert}{\rVert}
\DeclarePairedDelimiter\abs{\lvert}{\rvert}

\usepackage{authblk}

\title{Convergence Rate Analysis of Ratio Consensus Algorithms with Column-Allowable Matrices\thanks{The research was supported by NRDI (National Research, Development and Innovation Office) grant KKP 137490.}}

\author[1,2]{Balázs Gerencsér}
\author[1,2]{Levente Szemerédi}

\affil[1]{HUN-REN Alfréd Rényi Institute of Mathematics, Budapest, Hungary}
\affil[2]{Eötvös Loránd University, Budapest, Hungary}

\date{}

\begin{document}

\maketitle

\begin{abstract}
We give almost sure convergence rate bounds of ratio consensus algorithms when the protocol can be reformulated to be linear updates of vector values on a possibly larger, augmented network. This is an improvement of the results of Gerencsér and Gerencsér from 2021 by allowing zero values on auxiliary nodes infinitely often which makes the technique applicable to a much larger family of algorithms.
\end{abstract}

\section{Introduction}\label{sec:intro}
In distributed computing one fundamental subroutine is distributed averaging of initial values along a network using local communication \cite{nedic2018distributed}. It can be a building block for example to calculate linear synopses, random sampling and quantile computation \cite{kempe2003gossip}. Real world networks can have various imperfections which were addressed in previous works step-by-step. Usually some structure is assumed on the communication network in terms of matrices describing the linear update of the vector of values. For handling asynchronous, directed communication - when doubly stochasticity cannot be assumed for the underlying matrix process - the \emph{push-sum} protocol was introduced by Kempe, Dobra and Gehrke in \cite{kempe2003gossip} and extended by Bénézit et al. \cite{benezit2010weighted} as \emph{weighted gossip}. Later it was referred as \emph{ratio consensus.} Most theory regarding ratio consensus discusses convergence without exact rates.

Gerencsér and Gerencsér in \cite{gerencser2021tight} proved that the rate of convergence was connected to the difference of the top two Lyapunov exponents of the process of transition matrices. However, their theorems need the transition matrices to be \emph{allowable} (at every step there are both in- and outgoing communication) and the process of transition matrices being \emph{sequentially primitive:} after an almost surely finite stopping time each node contains a portion of the original value of every other.

To handle other challenges with the communication, variants of the push-sum protocol were introduced. Hadjicostis et al. in \cite{hadjicostis2015robust} considered the cases when there can be packet losses which are not directly detected by neither the sender nor the receiver, and the one when the knowledge of outgoing edges is imprecise. As a solution, they provided algorithms using \emph{running sums} which made the communication protocol nonlinear (in terms of the updates of the value vectors). Using an augmented network they constructed a Markov Chain formulation, which was linear on the augmented network. However, using \emph{virtual buffers}, the condition of allowability of the transition matrices and sequential primitivity of the matrix process do not hold anymore, thus the results on the exact rate of convergence cannot be used.

Motivated by the example above, in this work we extend the results of Gerencsér and Gerencsér \cite{gerencser2021tight} to be able to handle the push-sum protocol when the allowability and sequential primitivity condition is weakened. Thus, we can connect the convergence rate of the running-sums push-sum protocol to the difference of the first two Lyapunov exponents of the linear process on the augmented network. 

The paper is organized as follows. First, we discuss the fundamental definitions, previous results and state our results in the rest of Section \ref{sec:intro}. In Section \ref{sec:lyap} we include a short summary on Lyapunov exponents. In Section \ref{sec:bounds_on_rate} we provide the proofs for our results concerning the bounds on the convergence rates. In Section \ref{sec:nontriv_gap} we prove that the bounds are nonvanishing using the Birkhoff contraction coefficient. In Section \ref{sec:num_res} we present some numerical results and we conclude in Section \ref{sec:concl}.

\subsection{Push-sum protocol}

The push-sum protocol can be used to calculate approximate averages over a network. Here we describe the formulation called weighted gossip by Bénézit et al. (\cite{benezit2010weighted}).

Let us have a directed graph (network) $G(V,E)$ on vertex (node) set $V$ of size $p=\lvert V\rvert$ and edge (connection) set $E.$ For a vertex $v\in V,$ let us denote the set of its out-neighbors $N_v=\{u\in V:(v,u)\in E\}.$ Initially, we have a real value vector $\vec x_0\in\mathbb R^p$ and a nonnegative and nonzero weight vector $\vec w_0\in\mathbb R^p_+.$ Our aim is to approximate the ratio $(\mathbf 1^\top\vec x_0)/(\mathbf 1^\top\vec w_0),$ where $\mathbf 1=(1,1,\ldots,1)^\top$ of length $p.$ If we choose $\vec w_0=\mathbf 1,$ the ratio corresponds to the averages of the coordinates of $\vec x_0,$ but if $\vec w_0=(1,0,0,\ldots,0)^\top,$ then we get the sum of the values as the ratio.

At each iteration step $n,$ we apply the nonnegative, possibly random, transition matrix $A_n\in\mathbb R^{p\times p}_+$ to have the next value and weight vectors: $$\vec x_n=A_n \vec x_{n-1}, \quad \vec w_n=A_n \vec w_{n-1}.$$

We say that \emph{ratio consensus holds} if for each node $v\in V$ the ratio $x_n^v/w_n^v$ converges to $(\mathbf 1^\top\vec x_0)/(\mathbf 1^\top\vec w_0)$ almost surely, where $x_n^v$ and $w_n^v$ are the value and weight on node $v$ after step $n,$ respectively. Bénézit et al. proved in \cite{benezit2010weighted}, that ratio consensus holds when $(A_n)_{n=1}^\infty$ is a stationary and ergodic sequence of stochastic matrices with positive diagonals, and the expected value $\mathbb EA_1$ is irreducible. These conditions are not very strong, this formulation can be used for push-sum variants using simple messages, see also in \cite{kempe2003gossip}.

Gerencsér and Gerencsér in \cite{gerencser2021tight} proved, in a slightly different format, the following theorem which connects the rate of convergence to the difference of the first two Lyapunov exponents. We will provide in Section \ref{sec:lyap} an overview of the necessary theory regarding the Lyapunov exponents of stochastic matrix processes, and in Subsection \ref{subsec:wdefs} the further technical terms used in the theorem below.

\begin{theorem}[{\cite[Theorems 12, 14, 16]{gerencser2021tight}}]\label{thm:gerencser_orig}
 Let $(A_n)_{n=1}^\infty$ be a strictly stationary, ergodic process of $p\times p$ matrices such that $\mathbb E\log\lVert A_1\rVert^+<\infty.$ Denote $M_n=A_nA_{n-1}\cdots A_1$ the products. Moreover assume that $A_n$ is nonnegative and allowable for all $n,$ the process $(A_n)_{n=1}^\infty$ is sequentially primitive, the first two Lyapunov-exponents are different: $\lambda_1>\lambda_2,$ and for any pair of row indices $i,j,$ and any column index $k,$ $M_n^{ik}/M_n^{jk}$ is subexponential. Then, ratio consensus holds and an explicit upper bound can be given for the rate of convergence as follows: for all $i=1,\ldots,p,$ $\vec x\in\mathbb R^p$ and $\vec w\in\mathbb R^p_+$ nonzero we have $$\limsup_{n\to\infty}\frac1n\log\left\lvert\frac{\vec e_i^\top M_n\vec x}{\vec e_i^\top M_n\vec w}-\frac{\vec v^1\vec x}{\vec v^1\vec w}\right\rvert\leq-(\lambda_1-\lambda_2)\quad\text{a.s.}$$ where $\vec e_i$ is the $i$th unit vector and $v^1$ is from the first order approximation of $M_n.$ Moreover, for all pairs of nonnegative vectors $(x,w)\in\mathbb R^p\times \mathbb R^p$ apart from a set of Lebesgue measure zero, it holds that $$\limsup_{n\to\infty}\frac1n\log\max_i\left\lvert\frac{\vec e_i^\top M_n\vec x}{\vec e_i^\top M_n\vec w}-\frac{\vec v^1\vec x}{\vec v^1\vec w}\right\rvert=-(\lambda_1-\lambda_2)\quad\text{a.s.}$$
\end{theorem}

Hadjicostis et al. in \cite{hadjicostis2015robust} considered variants, where the messages are more complicated: instead of sending the values directly, so called running sums are communicated. For the case of possible message losses at each step the sender sends the sum of the values which would have been sent through the whole communication process according to the push-sum protocol. Using this convention, occasional message losses are not posing an issue: the information is just delayed until a successful delivery occurs. However, this protocol cannot be written using linear transformations directly on the network. But by using an augmented network, where some additional nodes representing buffer values are added and edges (connections) are modified accordingly, each step becomes a linear transformation of the values. However, the conditions of Theorem \ref{thm:gerencser_orig} are not satisfied for this matrix process.

\subsection{Weakened definitions}\label{subsec:wdefs}

We follow follow the notations of \cite{gerencser2021tight}. In particular, $(A_n)_{n=1}^\infty$ denotes the stochastic process, $M_n=A_nA_{n-1}\cdots A_1$ the product of the first $n$ terms -- i.e. the state after the $n^{th}$ step, $\lambda_i$ the $i^{th}$ Lyapunov exponent of the process $(A_n)_{n=1}^\infty$, $\bbR_+^p$ is the set of nonnegative and nonzero $p$-dimensional vectors.

The most problematic conditions from \cite{gerencser2021tight}, which are clearly not true for the matrix processes defined on the augmented network for the running sums version, are requiring the transition matrices to be allowable and the stochastic matrix process to be primitive. Instead, we show that when we use matrices not having zero columns, called column-allowable matrices, with the following weakening of primitivity, convergence rate results still can be deduced.

\begin{defi} \label{def:allowable}
A nonnegative matrix $A$ is called \emph{column-allowable} if each of its columns contains a strictly positive value. Similarly, it is called \emph{row-allowable} if each of its rows contains a strictly positive value. Moreover, $A$ is called \emph{allowable} if it is both column-allowable and row-allowable.
\end{defi}

\begin{defi} A strictly stationary process of nonnegative column-allowable matrices $(A_n)_{n=1}^\infty$ is \emph{weakly (forward) sequentially primitive} if there exists an almost surely finite stopping time $\tau,$ such that each row of the matrix $M_\tau=A_\tau A_{\tau-1}\cdots A_1$ is either strictly positive or zero.
\end{defi}

\begin{defi}\label{def:nodes} Let $(A_n)_{n=1}^\infty$ be a (random) sequence of nonnegative column-allowable matrices. We associate the rows and column indices with the nodes of an augmented network as follows. We call the subset of nodes $$\left\{i\in\{1,\dots,p\}\mid\bbP\left(\limsup_{n\to\infty}\left(\left\{\vec e_i^\top M_n=\vec 0\right\}\right)\right)=0\right\}$$ the set of \emph{real nodes}. The remaining nodes are called \emph{virtual nodes}.

We say that there is an \emph{edge from real node $i$ to real node $j$} if $\bbP\left(\limsup_{n\to\infty}\left(\left\{{\vec e}_j^\top A_n{\vec e}_i>0\right\}\right)\right)=1.$
\end{defi}

To state our results and formulate the proofs, we need the following notion of weak subexponentiality.

\begin{defi}
 A stochastic process $(\xi_n)_{n=1}^\infty$ is called \emph{weakly subexponential} if for any $\eps>0$ we have for all $n$, with finitely many exceptions, a.s. $|\xi_{n}|\leq e^{\eps n}$ or $\xi_n=\infty$. Equivalently, $\limsup_{\substack{n\to\infty,\\\xi_n\neq\infty}}\frac1{n}\log\abs{\xi_{n}}\leq 0.$ When we do not allow the process to take infinity as value but $\limsup_{n\to\infty}\frac1{n}\log\abs{\xi_{n}}\leq 0$ still holds, we call it \emph{subexponential}.
\end{defi}

\begin{cond}
 \label{cond:oseledec}
 Assume that $(A_n)_{n=1}^\infty$ is a strictly stationary ergodic process of $p\times p$ matrices such that $\bbE\log\norm{A_1}^+<\infty.$
\end{cond}

\begin{cond}[{Generalization of \cite[Condition 26]{gerencser2021tight}}]
 \label{cond:subexp_elements} Let $(A_n)_{n=1}^\infty$ be a sequence of (possibly deterministic) $p\times p$ matrices and $M_n=A_nA_{n-1}\cdots A_1$. Then, let us assume that for any pair of row indices $i,j$, and any column index $k$, it holds that $M_n^{ik}/M_n^{jk}$ is weakly subexponential.
\end{cond}

Note that if $(A_n)_{n=1}^\infty$ is a weakly sequentially primitive sequence, after time $\tau$ from the definition of weak sequential primitivity is reached, the indices $n$ for which the quantity is infinity only depend on the row indices $i$ and $j,$ and independent of the column index $k.$

\subsection{Main theorems}

We now arrive to state our main theorems. First, the one that quantifies at what rate the value vector and the weight vector become asymptotically collinear.

\begin{theorem}[{Generalization of \cite[Theorem 8]{gerencser2021tight}}]
\label{thm:general_total_var}
Let $(A_n)_{n=1}^\infty$ be a sequence of nonnegative matrices which is weakly sequentially primitive. Moreover, assume that Condition \ref{cond:oseledec} holds.
Then, for all pairs $(\vec x,\vec w)\in\bbR^p_+\times\bbR^p_+$, except for a set of Lebesgue measure zero, it almost surely holds that $$\lim_{n\to\infty}\frac1n\log\norm*{\frac{\vec x_n}{\1^\top \vec x_n}-\frac{\vec w_n}{\1^\top \vec w_n}}_{TV}=-(\lambda_1-\lambda_2)$$ where $\lambda_1$ and $\lambda_2$ are the first and the second largest Lyapunov exponents associated with $(A_n)$ and $\vec x_n=M_n\vec x,$ $\vec w_n=M_n\vec w$ for $M_n=A_nA_{n-1}\cdots A_1.$
In addition, for any fixed pair $(\vec x,\vec w)\in\bbR^p_+\times\bbR^p_+$ with strictly positive components with no exception, it holds that the above limit exists almost surely, and $$\lim_{n\to\infty}\frac1n\log\norm*{\frac{\vec x_n}{\1^\top \vec x_n}-\frac{\vec w_n}{\1^\top \vec w_n}}_{TV}\leq-(\lambda_1-\lambda_2).$$
\end{theorem}

The following two theorems give a bound on the rate of ratio consensus: the first in a more general case, the second in the more simple one for averaging.

\begin{theorem}[{Generalization of the end of \cite[Theorem 12,14,16]{gerencser2021tight}}]
 \label{thm:spectral_bound}
 Assume that the conditions of Theorem \ref{thm:general_total_var} are satisfied, in addition $\lambda_1-\lambda_2>0$, and Condition \ref{cond:subexp_elements} holds.
Take an arbitrary vector of initial values $\vec x\in\bbR^p$ and a nonnegative vector of initial weights $\vec w\in\bbR^p_+$ such that $\vec w\neq \vec 0.$ Then, ratio consensus takes place and an explicit upper bound for the rate of convergence can be given as follows: for all $i=1,\dots,p$, we almost surely have $$\limsup_{\substack{n\to\infty\\\vec e_i^\top M_n\neq\vec 0}}\frac1n\log\abs*{\frac{\vec e_i^\top M_n\vec x}{\vec e_i^\top M_n\vec w}-\frac{\vec v^1\vec x}{\vec v^1\vec w}}\leq-(\lambda_1-\lambda_2)$$ for some positive random vector $\vec v^1\in\bbR^p_+.$
\end{theorem}

\begin{theorem}[{Generalization of \cite[Theorem 19]{gerencser2021tight}}]
 \label{thm:col_stoch_spectral_bound}
 Assume that the conditions of Theorem \ref{thm:spectral_bound} hold, and $A_n$ is column stochastic for all $n$. Then, for any vector of initial values $\vec x\in\bbR^p$ and any nonnegative vector of initial weights $\vec w\in\bbR_+^p$ such that $\vec w\neq\vec 0$, we have almost surely for all $i=1,\dots,p$ $$\limsup_{\substack{n\to\infty\\\vec e_i^\top M_n\neq \vec 0}}\frac1n\log\abs*{\frac{\vec e_i^\top M_n\vec x}{\vec e_i^\top M_n\vec w}-\frac{\1^\top \vec x}{\1^\top \vec w}}\leq\lambda_2<0.$$
\end{theorem}

Finally, the following theorem states that the upper bound provided by the previous theorems is sharp.

\begin{theorem}[{Generalization of \cite[Theorem 21]{gerencser2021tight}}]
 \label{thm:strict_bound}
 Assume that the conditions of Theorem \ref{thm:spectral_bound} are satisfied. Then, for all pairs of nonnegative vectors $(\vec x,\vec w)\in\bbR_+^p\times\bbR_+^p$, such that $\vec x,\vec w\neq \vec 0$, expect perhaps for a set of Lebesgue measure zero, it almost surely holds that $$\limsup_{n\to\infty}\frac1n\log\max_{i,\text{ if } \vec e_i^\top M_n\neq\vec 0}\abs*{\frac{\vec e_i^\top M_n\vec x}{\vec e_i^\top M_n\vec w}-\frac{\vec v^1\vec x}{\vec v^1\vec w}}=-(\lambda_1-\lambda_2).$$
\end{theorem}

\section{Notes on Lyapunov exponents}
\label{sec:lyap}

In this section, we discuss some basic properties of the Lyapunov exponents, based on the articles \cite{oseledets1968multiplicative}, \cite{raghunathan1979proof}, and \cite{gerencser2021tight}.

Let $(A_n)_{n=1}^\infty$ be a (deterministic) sequence of $p\times p$ matrices. Then for any $\vec x\in\bbR^p$, the Lyapunov exponent of $\vec x$ w.r.t. the sequence $(A_n)_{n=1}^\infty$ is defined as $$\lambda(\vec x)=\limsup_{n\to\infty}\frac1n\log\abs{A_nA_{n-1}\cdots A_1\vec x}.$$

Then for any extended real number $-\infty\leq\mu\leq\infty$ let us define the set $$L_\mu=\{\vec x\in\bbR^p\mid\lambda(\vec x)\leq\mu\}.$$ $L_\mu$ is a linear subspace of $\bbR^p$, monotone increasing in $\mu$ and continuous from the right. Since there can only be finitely many strictly decreasing subspaces of $\bbR^p$, there are only finitely many possible values of the Lyapunov exponents: $\infty\geq\mu_1>\mu_2>\dots\mu_q\geq-\infty$ such that $$\bbR^p=L_{\mu_1}\supsetneq L_{\mu_2}\supsetneq\cdots\supsetneq L_{\mu_q}\supsetneq\{0\}=:L_{\mu_{q+1}}$$ and $L_\mu$ is a piecewise constant function of $\mu$ with discontinuities exactly at $\mu_i.$ So for $1\leq r\leq q,$ $x\in L_{\mu_r}\setminus L_{\mu_{r+1}}$ implies $\lambda(x)=\mu_r.$  Let us denote $i_r$ the dimension of the subspace $L_{\mu_r}$ with $1\leq r\leq q+1.$ Then, the codimension of $L_{\mu_{r+1}}$ relative to $L_{\mu_r}$ is $i_r-i_{r+1}$ which can be interpreted as the multiplicity of the Lyapunov exponent $\mu_r.$ Then, the full spectrum of Lyapunov exponents is $\infty\geq\lambda_1\geq\lambda_2\geq\cdots\geq\lambda_p\geq-\infty$ by setting $\lambda_i=\mu_r$ if $i_r\geq i>i_{r+1}$ for $1\leq i\leq p.$

When we get $(A_n)_{n=1}^\infty$ as a realization of a strictly stationary ergodic process, the following theorem holds.

\begin{theorem}[Oseledec's theorem \cite{oseledets1968multiplicative}]
 Under Condition \ref{cond:oseledec} the following holds. There exists a subset $\Omega'\subset\Omega$ with $\bbP(\Omega')=1$ such that for all $\omega\in\Omega'$ and for any $\vec x\in\bbR^p,$ the following limit exists: $$\lambda(\vec x)=\lim_{n\to\infty}\frac1n\log\abs{A_nA_{n-1}\dots A_1\vec x}.$$ Moreover the Lyapunov exponents $\lambda_1\geq\lambda_2\geq\dots\geq\lambda_p,$
 possibly taking the value $-\infty,$ do not depend on $\omega\in\Omega'$ either.
 The mapping $\omega\mapsto L_{\mu_r}(\omega)$
 is measurable from $\Omega$ to the Grassmannian manifold of linear subspaces of dimension $i_r.$ 
 In addition, we have with $M_n=A_nA_{n-1}\dots A_1$ the following limit almost surely exists: $$M^*=\lim(M_n^\top M_n)^{1/2n}$$
\end{theorem}

The singular value decomposition of the matrix $M_n$ gives $$M_n=U_n\Sigma_nV_n$$ where $U_n$ and $V_n$ are orthonormal matrices, and $\Sigma_n$ is a diagonal matrix with entries $\sigma_n^1\geq\sigma_n^2\geq\ldots\geq\sigma_n^p\geq0$ such that almost surely for all $k=1,\ldots,p$ we have $$\lambda_k=\lim_{n\to\infty}\frac1n\log\sigma_n^k$$ and the matrix $M_n$ is asymptotically of rank-$1$ which can be quantified as follows.

\begin{lemma}[{\cite[Lemma 5]{raghunathan1979proof}}]\label{lemma:first_order_approx}
Under Condition \ref{cond:oseledec} and assuming that there is a positive spectral gap, i.e. $\lambda_1>\lambda_2,$ the first row of the matrix $V_n$ converges almost surely to a (random) vector: $$\lim_{n\to\infty}\vec v_n^1=\vec v^1.$$
The matrix product $M_n=A_nA_{n-1}\dots A_1$ can be written as $$M_n=\vec u_n^1\vec v^1\sigma_n^1+O(e^{(\lambda_2+o(1))n}),$$ where $\vec u_n^1$ is the first column of the matrix $U_n.$
\end{lemma}

\section{Bounds on the convergence rate}\label{sec:bounds_on_rate}
When only minor modifications are needed to generalize the proofs, we do not write them out here in their full length.

In this Section, we provide the proofs for the theorems regarding the bounds on the convergence speed: Theorems \ref{thm:general_total_var}, \ref{thm:spectral_bound}, and \ref{thm:col_stoch_spectral_bound}. The main novelty is in the proofs of the lemmas required for the theorems, they are combined to finish the proofs in a similar way as in \cite{gerencser2021tight}.

We first establish two of these technical lemmas.

\begin{lemma}[{Generalization of \cite[Lemma 27]{gerencser2021tight}}]
\label{lemma:u_subexp}
 Under the Conditions \ref{cond:subexp_elements} and \ref{cond:oseledec}, and the additional assumption that $\lambda_1>\lambda_2$, it holds that for all $i$, $1/\vec u^{1i}_n$ is weakly subexponential almost surely. Moreover, if in Condition \ref{cond:subexp_elements} subexponentiality holds (i.e. $\infty$ is not allowed) then $1/\vec u^{1i}_n$ is subexponential almost surely.
\end{lemma}
\begin{proof}
 The structure of the proof is the same as of the original. Because there is a positive spectral gap, we almost surely have $M_n=\vec u_n^1\vec v^1\sigma_n^1+O(e^{(\lambda_2+o(1))n})$ by Lemma \ref{lemma:first_order_approx}. Take an arbitrary pair of row indices $i,j$, and compare the rows $M_n^j$ and $M_n^i.$ Choose a column index $k$ such that $\vec v^{1k}>0.$ Now we can write $$\frac{M^{jk}_n}{M^{ik}_n}=\frac{\vec u_n^{j1}\vec v^{1k}\sigma_n^1+O(e^{(\lambda_2+o(1))n})}{\vec u_n^{i1}\vec v^{1k}\sigma_n^1+O(e^{(\lambda_2+o(1))n})}=\frac{\vec u_n^{j1}+O(e^{(-\lambda_1+\lambda_2+o(1))n})}{\vec u_n^{i1}+O(e^{(-\lambda_1+\lambda_2+o(1))n})}.$$
 Indirectly suppose that $1/\vec u^{i1}_n$ is not weakly subexponential. Then for an $\eps>0$, we have with positive probability $\infty>1/\vec u^{i1}_n\geq e^{\eps n}$ along an infinite sequence, which is the same as $0<\vec u^{i1}_n\leq e^{-\eps n}.$ Because of the orthogonality of the matrix $U_n$ we can pick a $j$ such that $\infty>\vec u^{j1}_n>1/\sqrt p$ holds for an infinite subsequence. Combining these together we get that with positive probability on an infinite sequence the following inequality holds with some constant $C>0$: $$Ce^{\eps n}<\frac{M^{jk}_n}{M^{ik}_n}<\infty.$$ This is a contradiction to Condition \ref{cond:subexp_elements}.

For the second claim, supposing that $1/\vec u^{i1}_n$ is not subexponential gives $\infty\geq1/\vec u^{i1}_n\geq e^{\eps n}$. Combining this with the inequality $\infty>\vec u^{j1}_n>1/\sqrt p,$ we get $$Ce^{\eps n}<\frac{M^{jk}_n}{M^{ik}_n}\leq\infty.$$ holds with positive probability which gives a contradiction to the subexponential version of Condition \ref{cond:subexp_elements}.
\end{proof}

\begin{lemma}[{Generalization of \cite[Lemma 28]{gerencser2021tight}}]
 \label{lemma:v_pos}
 Under the conditions of Theorem \ref{thm:general_total_var}, with the additional assumption that $\lambda_1>\lambda_2$, and Condition \ref{cond:subexp_elements}, it holds that $v^{1i}>0$ for all $i=1,\dots,p.$
\end{lemma}
\begin{proof}
 In principle, the same proof works as in \cite{gerencser2021tight}. The difference is that we need to apply the weakly subexponential version of Lemma \ref{lemma:u_subexp} instead of Lemma 27 to the reversed process $(\bar A_n)$ and for the contradiction we also have to notice that the vector $v_n^1$ is strictly positive due to the Perron-Frobenius theorem.
\end{proof}

\subsection{Proof of Theorem \ref{thm:general_total_var}}

\begin{lemma}[{Generalization of \cite[Lemma 9]{gerencser2021tight}}]
\label{lemma:first_dominates}
Assume that Condition \ref{cond:oseledec} holds.
 In addition, assume that $(A_n)_{n=1}^\infty$ is nonnegative and \textbf{weakly sequentially primitive}. Then, there exists a subset $\Omega'\subset\Omega$ with $P(\Omega')=1$ such that for all $\omega\in\Omega'$ it holds that any strictly positive vector $\vec x>0,\vec x\in\bbR^p$ is contained in $\vec x\in L_{\mu_1}\setminus L_{\mu_2}:$ $$\lambda_1=\lim_{n\to\infty}\frac1n\log\abs{A_nA_{n-1}\cdots A_1\vec x}.$$

 Moreover, if we also assume that the conditions of Lemma \ref{lemma:v_pos} hold, we can relax the condition of $\vec x$ being strictly positive to be nonnegative and nonzero.
\end{lemma}
\begin{proof}
The proof has the same structure as of \cite{gerencser2021tight} but has multiple differences in the details.

Take $M_n^\top M_n=V_n^\top\Sigma_n^2 V_n,$ where the matrices $V_n$ and $\Sigma_n$ are from the singular value decomposition of the matrix $M_n=U_n\Sigma_n V_n.$ After a finite stopping time, this is a positive semidefinite matrix with strictly positive elements. Indeed, when each row of $M_n$ is either zero or strictly positive, the entries of product $M_n^\top M_n$ are strictly positive.

The eigenvalues of $M_n^\top M_n$ are the diagonal elements $(\sigma_n^i)^2$ of $\Sigma_n^2$ with corresponding eigenvectors $(\vec v_n^i)^\top.$ By the Perron-Frobenius theorem the eigenvalue of maximal modulus is unique and positive as is the corresponding eigenvector. So $(\sigma_n^1)^2$ is a single eigenvalue, and $\vec v_n^1>0$ elementwise.

Writing $\vec x$ in the basis defined by the rows of $V_n$ we have $\vec x^\top=\sum_i\alpha_n^i\vec v_n^i$ with $\alpha_n^i=\vec v_n^i\vec x.$ Then, $\vec x^\top M_n^\top M_n\vec x=\sum_i(\sigma_n^i)^2(\alpha_n^i)^2.$ Considering $$\alpha_n^1=\vec v_n^1\cdot \vec x=(\vec v_n^1-\vec v^1)\cdot \vec x+\vec v^1\cdot \vec x\to \vec v^1\cdot \vec x>0,$$ we have that $\alpha_n^1>\alpha_1>0$ holds for some positive $\alpha_1,$ possibly with finitely many exceptions either when $\vec x$ or $\vec v^1$ is strictly positive. Note that $v^1$ cannot be zero, because it is the (random) limit of vectors of length $1.$ Thus, $\vec x^\top M_n^\top M_n\vec x>(\sigma_n^1)^2\alpha_1^2,$ implying $\liminf_n\frac 1n\log\abs{M_nx}\geq\lambda_1,$ which concludes the proof.

\end{proof}

The following lemma from \cite{gerencser2021tight} is also needed to conclude the proof of Theorem \ref{thm:general_total_var}.

\begin{lemma}[{Generalization of \cite[Lemma 10]{gerencser2021tight}}]\label{lemma:exterior}
Suppose that Condition \ref{cond:oseledec} holds.
Consider the exterior product space $\bbR^p\wedge\bbR^p$ and the matrices $A_n\wedge A_n$ acting on it. Then, for all pairs $(\vec x,\vec w)\in\bbR^p\times\bbR^p,$ except for a set of Lebesgue measure zero, the a.s. limit $$\lim_{n\to\infty}\frac1n\log\abs{((A_nA_{n-1}\dots A_1)\wedge(A_nA_{n-1}\dots A_1))(\vec x\wedge \vec w)}$$ exists and is equal to $\lambda_1+\lambda_2.$
\end{lemma}

\begin{proof}[Proof of Theorem \ref{thm:general_total_var}]
The proof is the same as for \cite[Theorem 8]{gerencser2021tight} but we have to use Lemma \ref{lemma:first_dominates} instead of \cite[Lemma 9]{gerencser2021tight}.
\end{proof}

\subsection{Proof of Theorems \ref{thm:spectral_bound} and \ref{thm:col_stoch_spectral_bound}}

\begin{lemma}
 \label{lemma:ui1n_eimn} Let $(A_n)_{n=1}^\infty$ be a weakly sequentially primitive matrix series and $\tau$ as in the definition. Then, for any $n\geq\tau$ and $i\in1,\dots,p$, it holds that $\vec e_i^\top M_n=0$ and $\vec e_i^\top \vec u_n^1=0$ are equivalent.
\end{lemma}

\begin{proof}
Using the singular value decomposition, we get the following equation on row vectors: $$\vec e_i^\top M_n=\sum_{k=1}^p \vec e_i^\top \vec u_n^k\sigma_n^k \vec v_n^k.$$ The matrix $V_n$ is orthonormal and $\vec v^1_n\geq 0,$ so the span of $\{v_n^k,k\geq 2\}$ does not have a strictly positive element. Thus the LHS being strictly positive or zero is equivalent to the coefficient of $v_n^1$ being strictly positive or zero, which is $\vec e_i^\top \vec u_n^1\sigma_n^1.$ The matrix $M_n$ is nonzero so $\sigma_n^1>0$, thus the LHS being zero is equivalent to $\vec e_i^\top \vec u_n^1$ being zero.
\end{proof}

\begin{proof}[Proof of Theorem \ref{thm:spectral_bound}]
 The spectral gap is assumed to be positive so we can write $$\frac{\vec e_i^\top M_n\vec x}{\vec e_i^\top M_n\vec w}=\frac{\vec e_i^\top \vec u_n^1\vec v^1\vec x\sigma_n^1+O(e^{(\lambda_2+o(1))n})}{e_i^\top\vec u_n^1\vec v^1\vec w\sigma_n^1+O(e^{(\lambda_2+o(1))n})}=\frac{\vec e_i^\top\vec u_n^1\cdot\vec v^1\vec x+O(e^{(-\lambda_1+\lambda_2+o(1))n})}{\vec e_i^\top\vec u_n^1\cdot\vec v^1\vec w+O(e^{(-\lambda_1+\lambda_2+o(1))n})}.$$ By Lemma \ref{lemma:v_pos}
 we have $\vec v^1>0$ thus $\vec w\geq 0$ and $\vec w\neq\vec 0$ implies $\vec v^1\vec w>0.$ Applying \ref{lemma:ui1n_eimn} we have that along the sequence where we take the $\limsup$, $\vec e_i^\top\vec u_n^1\neq 0$ holds. Dividing both the numerator and the denominator by $\vec v^1\vec w$ and $\vec e_i^\top\vec u_n^1$ completes the proof with noting that $1/\vec e_i^\top\vec u_n^1=1/\vec u_n^{i1}$ is subexponential by Lemma \ref{lemma:u_subexp}.
\end{proof}

\begin{proof}[Proof of Theorem \ref{thm:col_stoch_spectral_bound}]
 Same as in \cite{gerencser2021tight}.
\end{proof}

\subsection{Exceptional set}

The coordinates of values and weights corresponding to real nodes, defined in Definition \ref{def:nodes}, are not spanning the whole $2p$-dimensional Eucledian space, so we need a stronger statement on the exceptional set than being  measure zero with respect to the $2p$-dimensional Lebesgue measure.

\begin{lemma}\label{lemma:exterior_exceptional}
Suppose that the conditions of Lemma \ref{lemma:exterior} hold. If additionally the conditions of Lemma \ref{lemma:v_pos} hold, then we can say the following about the exceptional set in Lemma \ref{lemma:exterior}.

Let $U^q\cong\bbR^q\leq\bbR^p$ be a $q$-dimensional linear subspace spanned by $q$ coordinate vectors. If there exists an almost surely finite stopping time $T$ such that the image of the matrix $M_T=A_T\cdots A_1$ is $U^q$ and its rank is maximal and shifting time by $T$ steps is a measure preserving transformation of the underlying probability space, then the zero-measure exceptional set of $\bbR^p\times\bbR^p$ cannot contain $U^q\times U^q$. In fact, inside $U^q\times U^q,$ the exceptional set is of $2q$-dimensional Lebesgue-measure zero.
\end{lemma}
\begin{proof}
 First, fix $\omega,$ the randomness in the matrix sequence. In $\bbR^p\times\bbR^p,$ the exceptional set is a linear subspace because it is defined with linear maps. The first part of Lemma \ref{lemma:exterior} states that the exceptional set is of Lebesgue-$2p$-measure zero in $\bbR^p\times\bbR^p$. Consider the map $(M_N,M_N)$ acting on $\bbR^p\times\bbR^p,$ its image is $U^q\times U^q$ with the appropriate inclusion. As we assumed that the rank of $M_N$ is $q$, the rank of the matrix $(M_N,M_N)$ is $2q$. The image of the exceptional set must be a proper subspace of $U^q\times U^q$, otherwise the exceptional set would have to be $\bbR^p\times\bbR^p.$ Shifting the time with $T$ - thus changing $\omega$ - we have that the intersection of the exceptional subspace corresponding to the shifted time and $U^q\times U^q$, is of $2q$-dimensional Lebesgue measure zero.

 The argument above works almost surely (in $\omega$) and applying Fubini's theorem similarly as before we get the statement.
\end{proof}

Note, that in Lemma \ref{lemma:exterior_exceptional}, the condition that the time shift by the stopping time $T$ is a measure preserving transformation is not automatic. It holds when the elements of the matrix process $(A_n)_{n=1}^\infty$ are independent, but it is not true for strictly stationary processes in general.

\subsection{Simpler conditions}
The Condition \ref{cond:subexp_elements} is hard to check in general, thus here we discuss some cases when it holds. For the proofs we follow \cite[Appendix D]{gerencser2021tight}.

\begin{nota}
 For a matrix series $(A_n)_{n=1}^\infty,$ we use the notations for the minimal and maximal positive elements of the matrix $A_n$ as follows: $$\alpha_n=\min_{ij}\{A_n^{ij}:A_n^{ij}>0\},\quad \beta_n=\max_{ij}A_n^{ij}.$$
\end{nota}

\begin{cond}[{Generalization of \cite[Condition 11]{gerencser2021tight}}]
 \label{cond:iid}
 Let $(A_n)_{n=1}^\infty$ be a process of i.i.d. random $p\times p$ nonnegative matrices. We assume that $\bbE\log^-\alpha_1>-\infty,$ where $\alpha_1$ is the minimal positive element of $A_1$ defined previously.
\end{cond}

\begin{defi}[{Generalization of \cite[Definition 4]{gerencser2021tight}}]
 For a weakly forward sequentially primitive series of matrices $(A_n)_{n=1}^\infty,$ the index of weakly forward sequential primitivity is defined as $$\psi_n=\min\{\psi\geq1:\text{each row of }A_{n+\psi-1}A_{n+\psi-2}\cdots A_n\text{ is either strictly positive or zero.}\}$$
 When the process $(A_n)_{n=1}^\infty$ is two-sided, the index of weakly backwards sequential primitivity is defined as $$\rho_n=\min\{\rho\geq1:\text{each row of }A_nA_{n+\psi-2}\cdots A_{n-\rho+1}\text{ is either strictly positive or zero.}\}$$
\end{defi}

\begin{cond}[{Generalization of \cite[Condition 13]{gerencser2021tight}}]
 \label{cond:bounded}
 Let $(A_n)_{n=1}^\infty$ be a strictly stationary, ergodic process of random, $p\times p$ nonnegative matrices. We say that $(A_n)_{n=1}^\infty$ is bounded from below and from above if there exist $\alpha,\beta>0$ such that we have a.s. $$0<\alpha\leq\alpha_n\leq\beta_n\leq\beta,$$ where $\alpha_n$ and $\beta_n$ are the minimal and maximal positive elements of the matrix $A_n$ defined before, and for the forward index of sequential primitivity $\psi_n,$ we have $\bbE\psi_n<\infty.$
\end{cond}

This condition is obviously satisfied if the range of $(A_n)_{n=1}^\infty$ is finite.

\begin{cond}[{Generalization of \cite[Condition 15]{gerencser2021tight}}]
 \label{cond:mmixing}
 Assume that for the index of forward sequential primitivity $\psi_n,$ we have $\bbE\psi_n<\infty.$ Furthermore, assume that $a_n=\log\alpha_n$ and $b_n=\log\beta_n$ satisfy a $q$th order $M$-mixing condition with some $q>4$ defined below.

 A two-sided strictly stationary process $(\xi_n)$ satisfies a $q$th order $M$-mixing condition with $q\geq 1,$ if $\bbE\abs{\xi_n}^q<\infty,$ and for any positive integer $N,$ we have, with some constant $C>0$ $$\bbE\abs{\sum_{n=1}N(\xi_n-\bbE\xi_n)}^q\leq CN^{q/2}.$$
\end{cond}

\begin{lemma}[{Generalization of \cite[Lemma 43]{gerencser2021tight}}]
 Let $M, B$ and $X$ be $p\times p$ matrices such that $M=BX$. Assume that each row of $B$ is either strictly positive or constant zero, and $X$ is a nonnegative, column-allowable matrix. Then, each row of $M$ is either strictly positive or constant zero - moreover, the index of the zero rows are the same for the matrices $B$ and $M$ -, and for any fixed pair of row indices $(i,j)$ and any column index $k$, we have $$\min_r\frac{B^{ir}}{B^{jr}}\leq\frac{M^{ik}}{M^{jk}}\leq\max_r\frac{B^{ir}}{B^{jr}}.$$
\end{lemma}
\begin{proof}
 When either $i$ or $j$ corresponds to a zero row, the inequality is trivial and equality holds. When both $i$th and  $j$th rows are strictly positive, then the proof is the same as in \cite{gerencser2021tight}.
\end{proof}

The next proposition shows, that instead of checking Condition \ref{cond:subexp_elements}, it is enough to guarantee one of the earlier ones.

\begin{prop}[{Generalization of \cite[Lemma 44]{gerencser2021tight}}]
 Either Condition \ref{cond:iid}, \ref{cond:bounded} or \ref{cond:mmixing} hold, Condition \ref{cond:subexp_elements} also holds.
\end{prop}

\begin{proof}
 The proof in \cite{gerencser2021tight} follows through. One need to notice that zero rows behave differently to positive rows but they make the inequalities trivial.
\end{proof}

\subsection{Bound is sharp}

\begin{proof}[Proof of Theorem \ref{thm:strict_bound}]
 Note that because of the weak sequential primitivity condition the maximum is always taken over a nonempty set. The original proof works well if under each $\max$ and $\min$ we have the additional requirement on the index $i$ not to let $\vec e_i^\top M_n=\vec 0.$
\end{proof}

\begin{remark}
 In Theorem \ref{thm:spectral_bound}, the strict positivity of the initial vectors is needed; thus by the end of the proof of Theorem \ref{thm:strict_bound} the "Lebesgue measure zero" set contains all the nonnegative but not strictly positive initial vectors. But this is enough for a counterexample.
 
 However, we want to use this generalized framework in cases when most coordinates of the initial vectors are zero (at the virtual node). Is this sharpness still true there?
\end{remark}

\section{Spectral gap is nontrivial}\label{sec:nontriv_gap}

For our results to be meaningful, we need to have that the spectral gap is nonzero. We address this question in this part. 
In \cite[Section VIII.]{gerencser2021tight}, the properties of the Birkhoff contraction coefficients are used. One application of this quantity is the proof of the Perron-Frobenius theorem. One such proof can be found in \cite{kohlberg1982contraction}. In that article, the contraction coefficient is analyzed in a more general setting: in terms of partially ordered vector spaces. One can identify a partial ordering of a vector space with a cone by taking the nonnegative vectors. If this cone is convex then the theory for the Birkhoff contraction coefficient holds. In the article \cite{gerencser2021tight}, the partial ordering is the one associated to the positive cone. But we can use the nonnegative cone instead, which is just the closure of the positive cone. Moreover, the interior of the nonnegative cone is the positive cone hence the positivity of transformations is the same: a \emph{positive matrix with respect to the nonnegative cone} can only contain positive elements.

We want to have an analogue of \cite[Theorem 36]{gerencser2021tight} with the condition of weak sequential primitivity. First, we discuss how the parts on the Birkhoff contraction coefficients can be addressed, after that we check the effect of the weak allowability. From now on, we follow the convention $\infty/\infty=1$.

\begin{defi}
 Let $\vec x,\vec y\in\bbR_+^p$ be nonnegative, nonzero vectors. Then their \emph{Hilbert distance} is defined as $$h(\vec x,\vec y)=\log\frac{\inf\{\lambda\geq 0\mid \lambda \vec y-\vec x\geq 0\}}{\sup\{\lambda\geq 0\mid \vec x-\lambda \vec y\geq 0\}}.$$
\end{defi}

The Hilbert distance is almost a metric. The difference is that it can take the value $\infty$ and it is zero whenever $\vec y=c\vec x$ for some positive real number $c>0$. It is easy to see that the Hilbert distance of the vectors $\vec x$ and $\vec y$ is $\infty$ if and only if for some index $k$ exactly one of $x_k$ and $y_k$ is zero -- we call this as \emph{$\vec x$ and $\vec y$ are not on the same face of $\bbR_+^n.$} When $\vec x$ and $\vec y$ are on the boundary of $\bbR_+^n$ but on the same face then their Hilbert distance coincides with the appropriate lower dimensional Hilbert-distance. With this in mind, we have the following definition.

\begin{defi}
The \emph{Birkhoff contraction coefficient} of a nonnegative matrix $A$ is defined as $$\tau(A)=\sup\left\{\frac{h(A\vec x,A\vec y)}{h(\vec x,\vec y)} \middle| \vec x,\vec y\in\bbR_+^p, 0<h(\vec x,\vec y) \right\}.$$
\end{defi}

The Hilbert distance defined above can be calculated directly from the coordinates of the vectors via the following lemma.

\begin{lemma}
 The Hilbert distance of two nonnegative, nonzero vectors $\vec x,\vec y\in\bbR_+^p$ which are on the same face can be calculated as $$h(\vec x,\vec y)=\log\max_{k,l\text{ nontriv}}\left(\frac{x_k}{y_k}/\frac{x_l}{y_l} \right),$$ where "$k,l\text{ nontriv}$" means that $\vec x$ and $\vec y$ are not both zero on coordinates $k$ or $l$. If the two vectors are not on the same face, their Hilbert distance is $\infty.$
\end{lemma}
\begin{proof}
 Since $\vec x$ and $\vec y$ are on the same face, we can apply the theorem for the appropriate dimensional Hilbert distance on open cones.
\end{proof}

For the better understanding of the Birkhoff contraction coefficient, one can turn to Birkhoff's theorem. Let us introduce the following quantity: $\varphi(A)=\max_{i,j}h(A^i,A^j)$ where $A^i$ and $A^j$ denote the $i$th and $j$th column vectors of the matrix $A$, respectively. Naturally, if $A$ has a row which has both positive and zero values, then $\varphi(A)=\infty$ otherwise it is a nonnegative real number. Using the previous lemma, we can calculate $\varphi(A)$ directly from the elements of $A$: if $A$ has a row which contain both positive and zero values then $\varphi(A)=\infty$, otherwise $$\varphi(A)=\log\max_{\substack{i,j,\\k,l\text{ nontriv}}}\left(\frac{A^{ki}}{A^{kj}}\right)/\left(\frac{A^{li}}{A^{lj}}\right),$$ where $k,l\text{ nontriv}$ means that $A_k$ and $A_l$ are both nonzero (strictly positive) rows of $A$. Birkhoff's theorem states (see \cite{birkhoff1957extensions}, \cite[Theorem 3.12]{seneta2006non} or \cite{cavazos2003alternative}) that if $A$ is a strictly positive matrix, then $$\tau(A)=\tanh\left(\frac{\varphi(A)}4\right)=\frac{e^{\varphi(A)/4}-e^{-\varphi(A)/4}}{e^{\varphi(A)/4}+e^{-\varphi(A)/4}}.$$

 The extension of Birkhoff's theorem to the nonnegative matrices is also proved in \cite{seneta2006non}, but here we give a different argument in the following two lemmas.

\begin{lemma} \label{lemma:nonneg_both_pos_and_zero}
 Let $A$ be a nonnegative matrix which has a row containing both at least one positive value and one zero but the matrix $A$ does not have a zero column. Then the contraction coefficient $\tau(A)=1.$
\end{lemma}

\begin{proof}
 It is enough to show a sequence of vector pairs $(\vec x_n,\vec y_n)_{n=1}^\infty$ such that the Hilbert distances $h(\vec x_n,\vec y_n)$ are positive and the ratio $\frac{h(A\vec x_n,A\vec y_n)}{h(\vec x_n,\vec y_n)}$ tends to $1$ as $n$ tends to infinity.

 We can reindex the rows and columns of the matrix $A$ in such a way that the special condition holds for the first row and the first element is strictly positive and the second one is zero: $A^{11}>0$ and $A^{12}=0.$

 Define the vectors $\vec x_n$ and $\vec y_n$ as follows: $\vec x_n^\top=(1, n+1, 0, \dots,  0 )$ and $\vec y_n^\top=(1, n, 0, \dots, 0).$ These vectors are on the same face of $\bbR_+^p,$ so their Hilbert-distance is finite: $$h(\vec x_n,\vec y_n)=\log\left(1+\frac1n\right).$$ To calculate the Hilbert distance $h(A\vec x_n,A\vec y_n),$ we have to take a maximum over pairs of coordinates. Here, there are two different cases: either the first coordinate is chosen or not.

 For a fixed $n$, when we consider the first coordinate with the $k$th, we have $$\gamma^k:=\log\left(\frac{A^{k1}+(n+1)\cdot A^{k2}}{A^{k1}+n\cdot A^{k2}}\right).$$ When both coordinates are different to $1$, let say, $k$ and $l$, we have $$\delta^{kl}=\log\left(\frac{A^{k1}+(n+1)\cdot A^{k2}}{A^{k1}+n\cdot A^{k2}}/\frac{A^{l1}+(n+1)\cdot A^{l2}}{A^{l1}+n\cdot A^{l2}}\right).$$ But $\gamma^k\geq 0$ and $\delta^{kl}=\gamma^k-\gamma^l,$ so the maximum must be $\gamma^k$ for some $k.$ From now on, we denote $k$ the index such that $\gamma^k$ is maximal, i.e. $h(A\vec x_n,A\vec y_n)=\gamma^k.$ We can rewrite $\gamma$ as follows: $$\gamma^k=\log\left(1+\frac1{A^{k1}/A^{k2}+n}\right).$$ Because we assumed that $A$ does not have a zero column we have that $A^{k2}>0,$ so we can do the division. From the previous form of $\gamma^k,$ it is clear that $k$ is independent of $n$ -- more precisely, we can choose the same $k$ for all $n$.

 Based on the previous calculations, we have that $$\frac{h(A\vec x_n,A\vec y_n)}{h(\vec x_n,\vec y_n)}=\frac{\log\left(1+\frac1{A^{k1}/A^{k2}+n}\right)}{\log\left(1+\frac1n\right)}$$ with $k$ chosen as before. This sequence tends to $1$ as $n$ tends to infinity as we wanted.
\end{proof}

\begin{lemma} \label{lemma:noneg_birkh}
 The equality $$\tau(A)=\tanh\left(\frac{\varphi(A)}4\right)$$ holds when $A$ is a nonnegative matrix. In particular, it is strictly less than $1$, if and only if each row of $A$ is either strictly positive or constant zero.
\end{lemma}

\begin{proof}
 Directly from the definition of $\tau$ and $\varphi$, it is clear that if any row of $A$ has both a zero and a positive element then $\varphi(A)=\infty$ and $\tau(A)=1$ by Lemma \ref{lemma:nonneg_both_pos_and_zero}. Thus, the equality holds.

 The other part of the proof is by a perturbation argument. First, notice that if $\vec x_\alpha=(\alpha x_2, x_2, \dots, x_p)^\top$ and $\vec y_\alpha=(\alpha y_2, y_2, \dots, y_p)^\top$ then for all $\alpha\geq 0$ the value $h(\vec x_\alpha,\vec y_\alpha)$ is the same because in the definition of the function $h$ the $\inf$ and the $\sup$ is taken over the same set no matter the value of $\alpha$. In particular, the ratio for the first and the second coordinates are the same, so when $\alpha=0$ it is still considered.

 Let us suppose that $A$ has $l$ zero rows and we have proven the equality up to $l-1$ zero rows. Moreover, the first row of $A$ is zero and the second is positive. Let $A_\alpha=(\alpha A_2^\top,A_2^\top,\dots,A_p^\top)^\top$ where $A_i$ is the $i$th row of $A$. In words: take $A$ and make the first row to be $\alpha$ times the second one. Now, for any column vertices $i$ and $j$ the value $h(A_\alpha^i,A_\alpha^j)$ is constant in $\alpha\geq 0$ thus $\varphi(A_\alpha)$ is also independent of $\alpha\geq 0$. On the other hand, for any $\vec x$ and $\vec y$ nonnegative vectors, the structure of $A_\alpha \vec x$ and $A_\alpha \vec y$ are the same as before, so $h(A_\alpha \vec x,A_\alpha \vec y)$ is also independent of $\alpha\geq 0$. Thus $\tau(A_\alpha)$ is independent of $\alpha\geq0$. Combining the fact that $\tau(A_\alpha)=\tanh(\frac{\varphi(A_\alpha)}4)$ for $\alpha>0$ and both the LHS and RHS is constant for $\alpha\geq 0$ we get that equality holds for $\alpha=0$ which means that the relation is also true for $A=A_0$.

Note that we used the continuity of $h$ only in some directions of $\bbR^p\times\bbR^p$. It is not continuous in the multivariable sense.
\end{proof}

The previous lemma generalizes the introductory part of \cite[Section VIII.]{gerencser2021tight} to our weaker settings. Submultiplicativity and $\tau(A)\leq 1$ hold as before. Now let us discuss what the consequences are.

\begin{lemma}[{Generalization of \cite[Lemma 33]{gerencser2021tight}}]
 \label{lemma:upperbound_tv}
 Let $\vec\xi$ and $\vec\eta$ be two probability vectors in $\bbR^p$ such that at each coordinate $i$, the values $\xi_i$ and $\eta_i$ are both zero or both positive. Then for their total variation distance, we have $$\norm{\vec\xi-\vec\eta}_{TV}\leq\frac12(e^{h(\vec\xi,\vec\eta)}-1).$$
\end{lemma}

\begin{proof}
 Discard those coordinates where the vectors are zero. This does not change neither the LHS nor the RHS. Then apply \cite[Lemma 33]{gerencser2021tight} in the appropriate dimension.
\end{proof}

\begin{theorem}[{Generalization of \cite[Theorem 32]{gerencser2021tight}}]
 Let $(A_n),n\geq 1$ be a strictly stationary, ergodic stochastic process of $p\times p$ matrices satisfying the conditions of Theorem \ref{thm:general_total_var}. Then $$\lambda_1-\lambda_2\geq-\bbE\log\tau(A_1).$$
\end{theorem}

\begin{proof}
 The proof is essentially the same as of the original theorem but we have to take care of the positivity constraints. To clarify things we take note on every step.

 Due to Lemma \ref{lemma:noneg_birkh} the inequality $$h(\vec x_n,\vec w_n)=h(A_nA_{n-1}\cdots A_1\vec x,A_nA_{n-1}\cdots A_1\vec w)\leq \prod_{k=1}^n\tau(A_k)\cdot h(\vec x,\vec w)$$ still holds in our weaker setting, i.e. when $A_m$ is weakly allowable for all positive integer $m$ and nonnegative vectors $\vec x,\vec w$. Of course, the inequality can be trivial.

 For the next step, we only consider those $(\vec x,\vec w)$ pairs when they are from the same face thus their Hilbert distance is finite. This means that almost surely $$\limsup_{n\to\infty}\frac1n\log h(x_n,w_n)\leq\limsup_{n\to\infty}\frac1n\sum_{k=1}^n\log\tau(A_k)=\lim_{n\to\infty}\frac1n\sum_{k=1}^n\log\tau(A_k)=\bbE\log\tau(A_1).$$

 For positive vectors $(\vec x,\vec w)$ we can apply Theorem \ref{thm:general_total_var} which implies that $\norm{\bar{\vec x_n}-\bar{\vec w_n}}_{TV}$ is exponentially small (where the bar operator means that the vector is normalized to be a probability vector -- this leaves the Hilbert distance unchanged). Combining this with Lemma \ref{lemma:upperbound_tv} we can write $\norm{\bar{\vec x_n}-\bar{\vec w_n}}_{TV}\leq h(\bar{\vec x_n},\bar{\vec w_n}).$ If $h$ is small then apply the lemma, if $h$ is big then it is naturally bigger then the total variation. Now we have for positive $(\vec x,\vec w)$ the almost sure inequality $$\limsup_{n\to\infty}\frac1n\log\norm{\bar{\vec x_n}-\bar{\vec w_n}}_{TV}\leq\limsup_{n\to\infty}\frac1n\log h(\bar{\vec x_n},\bar{\vec w_n})=\limsup_{n\to\infty}\frac1n\log h(\vec x_n,\vec w_n)\leq\bbE\log\tau(A_1).$$

 But by Theorem \ref{thm:general_total_var} for almost all positive pairs $(\vec x,\vec w)$ the LHS is equal to $-(\lambda_1-\lambda_2)$ almost surely not only with $\limsup$ but with $\lim$. Writing this up for such a pair we get the statement.

 Note that in the proof we used strictly positive starting vectors. But the statement is independent of the actual starting vectors of our process since the Lyapunov exponents depend only on the matrix series.
\end{proof}

From Theorem 34 to Corollary 38 in \cite{gerencser2021tight} every statement and proof is true word-by-word in the general setting if one use the generalized lemmas and theorems addressed before. In the proof of \cite[Theorem 36]{gerencser2021tight} we only have weak sequential primitivity which only guarantees that for a finite $m$ the matrix $M_m$ has only strictly positive or zero rows with positive probability which is enough for us in the general settings. We state here the conclusion of this chain of though.
\begin{theorem}[{Generalization of \cite[Corollary 38]{gerencser2021tight}}]
 Let $(A_n), n\geq 1$ be a strictly stationary, ergodic stochastic process of $p\times p$ matrices satisfying the conditions of Theorem \ref{thm:general_total_var}. Then, we have the following lower bound for the spectral gap almost surely: $$\limsup_{n\to\infty}\frac1n\log\tau(M_n)\geq-(\lambda_1-\lambda_2).$$
\end{theorem}
\begin{proof}
 Same as in \cite{gerencser2021tight}.
\end{proof}

In the proof of the reverse inequality we have to be careful about the zero rows of the product matrix.
\begin{theorem}[{Generalization of \cite[Theorem 39]{gerencser2021tight}}]
 Assume that the conditions of Theorem \ref{thm:spectral_bound} are satisfied. Then, we have almost surely $$\limsup_{n\to\infty}\frac1n\log\tau(M_n)\leq-(\lambda_1-\lambda_2).$$
\end{theorem}
\begin{proof}
 The proof is almost the same as in \cite{gerencser2021tight}. The difference is that we only have \cite[Equation (72)]{gerencser2021tight} for indices corresponding to strictly positive rows, but these are enough for \cite[Equation (73)]{gerencser2021tight} because of the different conditions on the index set of the $\max.$
\end{proof}

Combining the previous two claims together we get the following:
\begin{theorem}[{Generalization of \cite[Theorem 40]{gerencser2021tight}}]
 Assume that the conditions of Theorem \ref{thm:spectral_bound} are satisfied. Then, we have almost surely $$\lambda_1-\lambda_2=\lim_{n\to\infty}-\frac1n\log\tau(M_n).$$
\end{theorem}

\section{Numerical results}\label{sec:num_res}

Using our results, we have exact convergence rate for the solution of the packet dropping push-sum using the running sums methods, the natural question arise: how does the convergence rate ($r_\text{conv}$) depend on the packet size ($s$) or the rate of packet-drop ($r_\text{drop}$). The nature of the dependence of the two parameters are different. The first one corresponds to applying a (not necessary linear) map on the stochastic process elementwise, while the second one is a change of the measure on the underlying probability space.

The convergence rates are calculated for a network of size $30$ where each node has outdegree $10$ and plotted on Figure \ref{fig:dependence}. The numerical calculation of Lyapunov exponents is done by the QR method, introduced by Benettin, et. al. in \cite{benettin1980lyapunov_p1,benettin1980lyapunov_p2}.

\begin{figure}[h]
\begin{center}
\includegraphics[width=.7\textwidth]{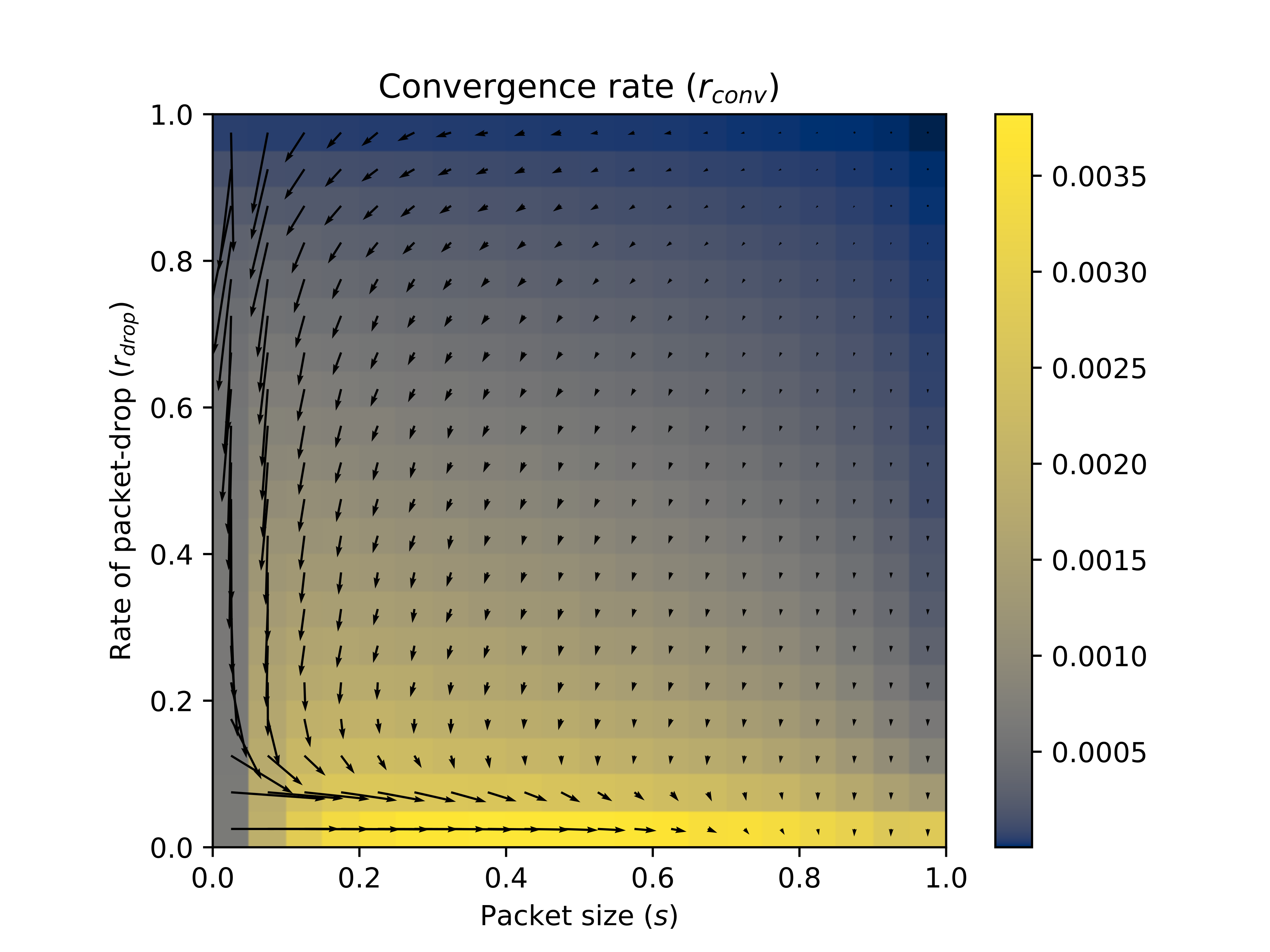}
\caption{Convergence rate for asynchronous push-sum with packet drop on a network of $30$ nodes. The approximation of the gradient is also displayed.}
\label{fig:dependence}
\end{center}
\end{figure}

We conjecture that the following properties are true for both functions $r_\text{conv}(s)$ and $r_\text{conv}(r_\text{drop})$ when we fix the other parameter:

\begin{enumerate}
 \item the function is continuous on the $(0,1)$ interval and
 \item for some $\varepsilon>0,$ the function is monotone on $[0,\varepsilon).$
\end{enumerate}

\section{Conclusion}\label{sec:concl}

In this paper, we provided a significant extension for the results in \cite{gerencser2021tight} in terms of usability by only requiring weak sequential primitivity for the matrix process. These results extend the family of analyzable push-sum like protocols where internal steps can be identified with linear transformations on an augmented network. For example the running sum algorithm can handled using this framework.

\printbibliography

\end{document}